\documentclass[12pt]{amsproc}

\usepackage{fancyhdr}

\author[Rodrigues]{Eliana Rodrigues}
\address{Department of Academic Areas, Instituto Federal de Goi\'as, Formosa-GO 73813-816, Brazil}
\email{eliana.rodrigues@ifg.edu.br}

\author[de Melo]{Emerson de Melo}
\address{Department of Mathematics, University of Bras\'ilia, Bras\'ilia-DF 70910-900, Brazil}
\email{emerson@mat.unb.br}

\author[Ercan]{Gülin Ercan}
\address{Department of Mathematics, Middle East Technical University,
06800, Ankara/Turkey}
\email{ercan@metu.edu.tr}


\title{NILPOTENT RESIDUAL\\ OF A FINITE GROUP }
\newtheorem{theo}{\sc Theorem}[section]
\newtheorem{lem}[theo]{\sc Lemma}
\newtheorem{proposition}[theo]{\sc Proposition}

\begin{document}

\begin{abstract}
Let $F$ be a nilpotent group
acted on by a group $H$ via automorphisms and let the group $G$ admit the
semidirect product $FH$ as a group of automorphisms so that $C_G(F) = 1$. We prove that  the order of $\gamma_\infty(G)$, the rank of $\gamma_\infty(G)$ are bounded in terms of  the orders of $\gamma_{\infty}(C_G(H))$ and $H$, the rank of $\gamma_{\infty}(C_G(H))$ and the order of $H$, respectively in cases where either  $FH$ is a Frobenius group; $FH$ is a Frobenius-like group satisfying some certain conditions;  or $FH=\langle \alpha,\beta\rangle$ is  a dihedral group generated by the involutions $\alpha$ and $\beta$ with  $F =\langle \alpha\beta\rangle$ and $H =\langle\alpha \rangle$. 

\end{abstract}

\maketitle

\section{Introduction}

Throughout all groups are finite. Let a group $A$ act by automorphisms on a group $G$. For any $a \in A$,  we denote by $C_G(a)$ the set $\{x\in G : x^a=x\},$ and write $C_G(A)=\bigcap_{a\in A}C_G(a).$ In this paper we focus on a certain question related to the strong influence of the structure of such fixed point subgroups on the structure of  $G$, and present some new results when the group $A$ is a Frobenius group or a Frobenius-like group or a dihedral group of automorphisms. 

In what follows we denote by $A^\#$ the set of all nontrivial elements of $A$, and we say that $A$ acts coprimely on $G$ if $(|A|,|G|)=1$. Recall that a Frobenius group $A=FH$ with kernel $F$ and complement $H$ can be characterized as a semidirect product of a normal subgroup $F$ by $H$ such that $C_F(h)=1$ for every $h \in H^\#$. Prompted by Mazurov's problem $17.72$ in the Kourokva Notebook \cite{KN}, some attention was given to the situation where a Frobenius group $A=FH$ acts by automorphisms on the group $G$. In the case where the kernel $F$ acts fixed-point-freely on $G$, some results on the structure of $G$ were obtained by Khukhro, Makarenko and Shumyatsky in a series of papers \cite{MS},  \cite{MKS}, \cite{K1}, \cite{K2}, \cite{K3}, \cite{KM1}, \cite{ENP}. They observed that various properties of $G$ are in a certain sense close to the corresponding properties of the fixed-point subgroup $C_G(H)$, possibly also depending on $H$.  In particular, when $FH$ is metacyclic they proved that if $C_G(H)$ is nilpotent of class $c$, then the nilpotency class of $G$ is bounded in terms of $c$ and $|H|$. In addition, they constructed examples showing that the result on the nilpotency class of $G$ is no longer true in the case of non-metacyclic Frobenius groups. However, recently in \cite{EJ2} it was proved that if $FH$ is supersolvable and $C_G(H)$ is nilpotent of class $c$, then the nilpotency class of $G$ is bounded in terms of $c$ and $|FH|$.

Later on, as a generalization of Frobenius group the concept of a Frobenius-like group was introduced by Ercan and Güloğlu in \cite{EG1}, and  their action  studied in a series of papers  \cite{EG2}, \cite{EGK1},\cite{EGK2},\cite{EGK3},\cite{EG4},\cite{EG5}.  A finite group $FH$ is said to be Frobenius-like if it has a nontrivial nilpotent normal subgroup $F$ with a
nontrivial complement $H$ such that $FH/F'$ is a Frobenius group with Frobenius kernel $F/F'$ and complement $H$ where $F'=[F,F]$. Several results about the properties of a
finite group $G$ admitting a Frobenius-like group of automorphisms $FH$ aiming at restrictions on $G$ in terms of $C_G(H)$ and focusing mainly on bounds for the Fitting height and related parameters as a generalization of earlier results obtained for
Frobenius groups of automorphisms; and also new theorems for Frobenius-like groups based on new representation-theoretic results. In these papers two special types of Frobenius-like groups have been handled. Namely, Frobenius-like groups $FH$ for which $F'$ is of prime order and is contained in $C_F(H)$; and  the Frobenius-like groups $FH$ for which $C_F(H)$ and $H$ are of prime orders, which we call Type I and  Type II, respectively throughout the remainder of this paper.

In \cite{PS} Shumyatsky showed that the techniques developed in \cite{ENP} can be 
used in the study of actions by groups that are not necessarily Frobenius. He considered a dihedral group $D=\langle \alpha, \beta \rangle$ generated by two involutions $\alpha$ and $\beta$ acting on a finite group $G$ in such a manner that $C_G(\alpha \beta)=1$. In particular, he proved that if $C_G(\alpha)$ and $C_G(\beta)$ are both nilpotent of class $c$, then $G$ is nilpotent and the nilpotency class of $G$ is bounded solely in terms of $c$. In \cite{EJ1},  a similar result was obtained for other groups. It should also be noted that in \cite{EG4} an extension of \cite{PS} about the nilpotent length obtained by proving that the nilpotent length of a group $G$ admitting  a dihedral group of automorphisms in the same manner is equal to the maximum of the nilpotent lengths of the subgroups
$C_G(\alpha)$ and $C_G(\beta)$. 

Throughout  we shall use the expression ``$(a,b,\dots )$-bounded'' to abbreviate ``bounded from above in terms of  $a, b,\dots$ only''.  Recall that the rank $\mathbf r(G)$ of a finite group $G$ is the minimal number  $r$  such that every subgroup of $G$ can be generated by at most $r$ elements. Let $\gamma_\infty(G)$ denote the \textit{nilpotent residual} of the group $G$, that is the intersection of all normal subgroups of $G$ whose quotients are nilpotent.  Recently, in \cite{EAP}, de Melo, Lima and Shumyatsky considered the case where $A$ is a finite group of prime exponent $q$ and of order at least $q^3$ acting on a finite $q'$-group $G$.  Assuming that $|\gamma_\infty(C_G(a))| \leq m$ for any $a \in A^\#$,  they showed that $\gamma_\infty(G)$ has $(m,q)$-bounded order. In addition, assuming that the rank of  $\gamma_\infty(C_G(a))$ is at most $r$ for any $a \in A^\#$, they proved that the rank of  $\gamma_\infty(G)$ is $(m,q)$-bounded. Later, in \cite{E}, it was proved that the order of $\gamma_\infty(G)$ can be bounded by a number independent of the order of $A$. 

The purpose of the present article is to study the residual nilpotent of finite groups admitting a Frobenius group, or a Frobenius-like group of Type I and Type II, or a dihedral group as a group of automorphisms. Namely we obtain the following results.

\textbf{Theorem A} 
Let $FH$ be a Frobenius, or a Frobenius-like group of Type I or Type II,  with kernel $F$ and complement $H$. 
Suppose that $FH$ acts on a finite group $G$ in such a way that $C_G(F)=1$. Then
\begin{itemize}
\item[a)]  $|\gamma_\infty(G)|$ is bounded solely in terms of $|H|$ and $|\gamma_{\infty}(C_G(H))|$;
\item[b)] the rank of $\gamma_\infty(G)$ is bounded in terms of $|H|$ and the rank of $\gamma_{\infty}(C_G(H))$.
\end{itemize}

\textbf{Theorem B} 
Let $D= \langle\alpha, \beta \rangle$ be a dihedral group generated by two involutions $\alpha$ and $\beta$. Suppose that $D$ acts on a finite group $G$ in such a manner that $C_G(\alpha\beta)=1$. Then 

\begin{itemize}
\item[a)] $|\gamma_\infty(G)|$ is bounded solely in terms of $|\gamma_{\infty}(C_G(\alpha))|$ and $|\gamma_\infty(C_G(\beta))|$;
\item[b)] the rank of $\gamma_\infty(G)$ is bounded in terms of the rank of $\gamma_{\infty}(C_G(\alpha))$ and $\gamma_\infty(C_G(\beta))$. 
\end{itemize}

The paper is organized as follows. In Section 2 we list some results to which we appeal frequently. Section 3 is devoted to the proofs of  two key propositions which play crucial role in proving Theorem A and Theorem B whose proofs  are given in Section 4.

\section{Preliminaries}

If $A$ is a group of automorphisms of $G$, we use $[G,A]$ to denote the subgroup generated by elements of the form $g^{-1}g^a$, with $g \in G$ and $a \in A$. Firstly, we recall some well-known facts about coprime
action, see for example \cite{GO}, which will be used  without any further references.

\begin{lem} \label{lema 1}
Let $Q$ be a group of automorphisms of a finite group $G$ such that $(|G|,|Q|) = 1$. Then
\begin{itemize}
\item[(a)] $G= C_G(Q)[G,Q]$. 
\item[(b)] $Q$ leaves some Sylow $p$-subgroup of $G$ invariant for each prime $p \in \pi(G)$.
\item[(c)] $C_{G /N}(Q) =  C_G(Q) N /N$ for any $Q$-invariant normal subgroup $N$ of $G$.
\end{itemize}
\end{lem}

We list below some facts about  the action of Frobenius and Frobenius-like groups. Throughout, a non-Frobenius Frobenius-like group is always considered under the hypothesis below.

\textbf{Hypothesis*} Let $FH$ be a non-Frobenius Frobenius-like group with kernel $F$ and complement $H$. Assume that  a Sylow $2$-subgroup of $H$ is cyclic and normal, and $F$ has no extraspecial sections of order $p^{2m+1}$ such that $p^m + 1 = |H_1|$ for some subgroup $H_1 \leq H$. 
 
It should be noted that Hypothesis* is automatically satisfied if either $|FH|$ is odd or $|H| = 2$.

\begin{theo}\label{theFrob}
Suppose that a finite group $G$ admits a Frobenius group or a Frobenius-like group of automorphisms $FH$ with kernel F and complement H such that $C_G(F)=1$. Then $C_G(H)\ne 1$ and $\mathbf r(G)$ is bounded in terms of $\mathbf r(C_G(H))$ and $|H|.$
\end{theo}

\begin{proposition}\label{Fr}
Let $FH$ be a Frobenius, or a Frobenius-like group of Type I or Type II. Suppose that $FH$ acts on a $q$-group $Q$ for some prime $q$ coprime
to the order of $H$ in case $FH$ is not Frobenius. Let $V$ be a $kQFH$-module where $k$ is a field with characteristic not dividing $|QH|.$ Suppose further that $F$ acts fixed-point freely on the semidirect
product $VQ$. Then we have $C_V(H)\ne 0$ and $$Ker(C_Q(H) \ \textrm{on} \ C_V(H)) = Ker(C_Q(H) \ \textrm{on} \ V).$$ \end{proposition}

\begin{proof} See \cite{EGO} Proposition 2.2 when $FH$ is Frobenius; \cite{EG2} Proposition C when $FH$ is Frobenius-like of Type I; and \cite{EG3} Proposition 2.1 when $FH$ is Frobenius-like of Type II. It can be easily checked that \cite{
EGO} Proposition 2.2 is  valid when $C_Q(F)=1$ without the coprimeness condition $(|Q|,|F|)=1.$
\end{proof}

The proof of  the following theorem can be found in \cite{PS} and in \cite{E1}.

\begin{theo}\label{theDih}
Let $D= \langle\alpha, \beta \rangle$ be a dihedral group generated by two involutions $\alpha$ and $\beta$. Suppose that $D$ acts on a finite group $G$ in such a manner that $C_G(\alpha\beta)=1$. Then
\begin{itemize}
\item[(a)] $G = C_G(\alpha)C_G(\beta)$;

\item[(b)]  the rank of $G$ is bounded in terms of the rank of $C_G(\alpha)$ and $C_G(\beta)$;

\end{itemize}
\end{theo}

\begin{proposition} \label{dih}Let $D =\langle \alpha,\beta\rangle$ be a dihedral group generated by the involutions $\alpha$ and $\beta.$ Suppose that $D$ acts on a $q$-group $Q$ for some prime $q$ and let V be a $kQD$-module for a field $k$ of characteristic different from
$q$ such that the group $F =\langle \alpha\beta\rangle$ acts fixed point freely on the semidirect
product $VQ$. If $C_Q(\alpha)$ acts nontrivially on $V$ then we have $C_V (\alpha)\ne 0$
and $Ker(C_Q(\alpha) \ \textrm{on} \ C_V(\alpha)) = Ker(C_Q(\alpha) \ \textrm{on} \ V)$. 
\end{proposition}
\begin{proof} This is Proposition C in \cite{EG4}.
\end{proof}
The next two results were established in \cite[Lemma 1.6]{KS2} .

\begin{lem}\label{sink1}
Suppose that a group $Q$ acts by automorphisms on a group $G$. If $Q=\langle q_1,\ldots , q_n \rangle$, then $[G,Q]=[G,q_1]\cdots [G,q_n].$
\end{lem}

\begin{lem}\label{sink2}
Let $p$ be a prime, $P$ a finite $p$-group and $Q$ a $p'$-group of automorphisms of $P$. 
\begin{itemize}
   \item  [a)] If $|[P,q]|\leq m$ for every $q\in Q$, then $|Q|$ and $|[P,Q]|$ are $m$-bounded.
\item[b)] If $r([P,q])\leq m$ for every $q\in Q$, then $r(Q)$ and $r([P,Q])$ are $m$-bounded.
\end{itemize}
\end{lem}

We also need the following fact whose proof can be found in \cite{CP}.

\begin{lem} \label{lema 3.1}
Let $G$ be a finite group such that $\gamma_\infty(G) \leq F(G)$. Let $P$ be a Sylow $p$-subgroup of $\gamma_\infty(G)$ and $H$ be a Hall $p'$-subgroup of $G$. Then $P= [P,H]$.
\end{lem}

\section{Key Propositions} We prove below a new proposition  which studies the actions of Frobenius and Frobenius-like groups and forms the basis in proving Theorem A. 
\\

\begin{proposition}\label{prop1}
Assume that $FH$ be a Frobenius group, or a Frobenius-like group of Type I or Type II with kernel $F$ and complement $H$. Suppose that $FH$ acts on a $q$-group $Q$ for some prime $q$. Let $V$ be an irreducible $\mathbb{F}_pQFH$-module where $\mathbb{F}_p$ is a field with characteristic $p$ not dividing $|Q|$ such  that $F$ acts fixed-point-freely on the semidirect product $VQ$.  Additionaly, we assume that  $q$ is coprime to $|H|$ in case where $FH$ is not Frobenius. Then $\mathbf  r([V,Q])$ is bounded in terms of   $\mathbf  r([C_V(H), C_Q(H)])$ and $|H|$.
\end{proposition}

\begin{proof} Let $\mathbf  r([C_V(H), C_Q(H)])=s.$ We may assume that $V=[V,Q]$ and hence $C_V(Q)=0$. 
By Clifford's Theorem, $V=V_1\oplus \cdots \oplus V_t$,  direct sum of 
 of $Q$-homogeneous components $V_i$ , which are transitively permuted by $FH$. Set  $\Omega =\{V_1,\dots, V_t\}$ and fix an $F$-orbit $\Omega_1$ in $\Omega$. Throughout, $W=\Sigma_{U\in \Omega_1}U.$
 
 Now, we split the proof into a sequence of steps.\\

{\it (1) We may assume that $Q$ acts faithfully on $V$.  Furthermore $Ker(C_Q(H) \ \textrm{on} \ C_V(H)) = Ker(C_Q(H) \ \textrm{on} \ V)=1$. }

\begin{proof} Suppose that $Ker(Q \ \rm{on} \ V)\neq 1$  and set $\overline{Q} =Q/Ker(Q \ \rm{on} \ V)$. Note that  since $C_Q(F)=1$,  $F$ is a Carter subgroup of $QF$ and hence also a Carter subgroup of  $\overline{Q}F$ which implies that $C_{\overline{Q}}(F)=1$. Notice that  the equality $\overline{C_Q(H)}=C_{\overline{Q}}(H)$ holds in case $FH$ is Frobenius (see \cite{ENP} Theorem 2.3). The same equality holds in case where $FH$ is non-Frobenius due to the coprimeness condition $(q,|H|)=1.$ Then  $[C_V(H),C_Q(H)]=[C_V(H),C_{\overline{Q}}(H)]$ and so we may assume that $Q$ acts faithfully on $V$. 
Notice that  by Proposition \ref{Fr} we have
$$Ker(C_Q(H) \ \textrm{on} \ C_V(H)) = Ker(C_Q(H) \ \textrm{on} \ V)=1$$ establishing the claim. 
\end{proof}

{\it (2) We may assume that $Q=\langle c^F \rangle$ for any nonidentity element $c\in C_{Z(Q)}(H)$ of  order $q$. In particular $Q$ is abelian.} 

\begin{proof} We obtain that $C_{Z(Q)}(H)\ne 1$ as $C_Q(F)=1$ by Proposition \ref{Fr}.  Let  now $1\ne c \in C_{Z(Q)}(H)$ of order $q$ and consider $\langle c^{FH} \rangle=\langle c^F \rangle$, the minimal $FH$-invariant subgroup containing $c$. Since $V$  is an irreducible $QFH$-module on which $Q$ acts faithfully we have that $V=[V,\langle c^F \rangle]$. Thus we may assume that $Q=\langle c^F \rangle$ as claimed. 
\end{proof}

{\it (3) $V=[V,c]\cdot [V,c^{f_1}] \cdots[V,c^{f_n}]$ where $n$ is a $(s,|H|)$-bounded number. Hence it suffices to bound $\mathbf r([W,c])$. }

\begin{proof}
Notice that the group $C_Q(H)$ embeds in the automorphism group of $[C_V(H),C_Q(H)]$ by step $\it(1)$. Then $C_Q(H)$ has $s$-bounded rank by Lemma  \ref{sink2}. This yields by Theorem \ref{theFrob}  that $Q$  has $(s,|H|)$-bounded rank. Thus, there exist $f_1=1,\ldots ,f_n$ in $F$ for an $(s,|H|)$-bounded number $n$ such that  $Q=\langle  c^{f_1},\ldots,c^{f_n} \rangle$. Now $V=[V,c]\cdot [V,c^{f_2}] \cdots[V,c^{f_n}]=\prod_{i=1}^n [V,c]^{f_i}$ by Lemma \ref{sink1}.  This shows that we need only to bound $\mathbf r([V,c])$ suitably. In fact it suffices to show that $\mathbf r([W,c])$ is suitably bounded as $V=\Sigma_{h\in H}W^h.$
\end{proof}

\textit{(4) $H_1=Stab_H(\Omega_1)\ne1$. Furthermore the rank of the sum of members of $\Omega_1$ which are not centralized by $c$ and contained  in a regular $H_1$-orbit, is suitably bounded.}

\begin{proof} Fix  $U\in \Omega_1$ and set $Stab_F(U)=F_1$. Choose a transversal $T$ for $F_1$ in $F.$ Let   $W=\sum_{t\in T}U^{t}$ where $T$ is a transversal for $F_1$ in $F$ with $1\in T.$ Then we have $V=\sum_{h\in H}W^h$. Notice  that $[V,c]\ne 0$ by $\it(1)$ which implies that $[W,c]\neq 0$ and hence $[U^t,c]=U^t$ for some $t\in T$. Without loss of generality we may assume that $[U,c]=U.$ 

Suppose that $Stab_H(\Omega_1)=1$. Then we also have $Stab_H(U^t)=1$ for all $t\in T$  and hence the sum $X_t=\sum_{h\in H}U^{th}$ is direct for all $t\in T.$  Now, $U\leq X_1$.  It holds that $$C_{X_t}(H)=\{ \sum_{h\in H}v^{h} \ : \ v\in U^t\}.$$ Then  $|U|=|C_{X_1}(H)|=|[C_{X_1}(H),c]|\leq |[C_V(H),C_Q(H)]|$ implies $\mathbf  r(U)\leq s.$ On the other hand $V=\bigoplus_{t\in T}X_t$ and $$[C_V(H),c]=\bigoplus \{ [C_{X_t}(H),c] : t\in T \,\, \text{with }\,\,[U^t,c]\ne 0\}\leq [C_V(H),C_Q(H)].$$ In particular, $\{t\in T : [U^t,c]\ne 0\}$ is suitably bounded whence $\mathbf  r([W,c])$ is $(s,|H|)$-bounded.  Hence we may assume that $Stab_H(\Omega_1)\ne 1.$

Notice that every element of a regular $H_1$-orbit in $\Omega_1$ lies in a regular $H$-orbit in $\Omega$. Let  $U\in \Omega_1$ be contained in a regular $H_1$-orbit of $\Omega_1.$ Let $X$ denote the sum of the members of the $H$-orbit of $U$ in $\Omega$, that is $X=\bigoplus_{h\in H}U^h$.  Then $C_X(H)=\{ \sum_{h\in H}v^{h} \ : \ v\in U\}$. If $[U,c]\ne 0$ then by  repeating the same argument in the above paragraph we show that  $\mathbf  r(U)\leq s$ is suitably bounded. On the other hand  the number, say $m$,  of all $H$-orbits in $\Omega$ containing a member $U$ such that $[U,c]\ne 0$ is suitably bounded because $m\leq \mathbf r([C_V(H),c])\leq s.$  It follows then that the rank of the sum of members of $\Omega_1$ which are not centralized by $c$ and contained  in a regular $H_1$-orbit, is suitably bounded. 
\end{proof}

{\it (5) We may assume that $FH$ is not Frobenius.}

\begin{proof} Assume the contrary that $FH$ is Frobenius.  Let $H_1=Stab_H(\Omega_1)$ and pick $U\in \Omega_1$. Set $S=Stab_{FH_1}(U)$ and $F_1=F\cap S$. Then $|F:F_1|=|\Omega_1|=|FH_1:S|$ and so $|S:F_1|=|H_1|$. Since $(|F_1|,|H_1|)=1$, by the Schur-Zassenhaus theorem there exists a complement, say $S_1$ of $F_1$ in $S$ with $|H_1|=|S_1|$. Therefore there exists a conjugate of $U$ which is $H_1$-invariant. There is no loss in assuming that $U$ is $H_1$-invariant.  On the other hand if  $1\neq h\in H_1$ and $x\in F$ such that $U^{xh}=U^x$, then $[h,x]\in Stab_{F}(U)=F_1$ and so $F_1x=F_1x^h=(F_1x)^h$. This implies that $ F_1x\cap C_F(h)$ is nonempty. Now the Frobenius action of $H$ on $F$ forces that $x\in F_1$. This  means that for each $x\in F\setminus F_1$ we have $Stab_{H_1}(U^x)=1$. Therefore $U$ is the unique member of $\Omega_1$ which is $H_1$-invariant and all the $H_1$-orbits other than  $\{U\}$ are regular.  By $\it(4)$,  the rank of the sum of all members of $\Omega_1$ other than $U$ is is suitably bounded. In particular $\mathbf r(U)$ and hence $\mathbf r([W,c])$ is suitably bounded in case where $[U^x,c]\ne 0$ for some $x\in F\setminus F_1$. Thus we may assume that $c$ is trivial on $U^x$ for all $x\in F\setminus F_1$. Now we have  $[W,c]=[U,c]=U.$

Due to the action by scalars of the abelian group $Q$ on $U$, it holds that $[Q,F_1]\leq C_Q(U)$. We also know that $c^x$ is trivial on $U$ for each $x\in F\setminus F_1$.  Since $C_Q(F)=1$, there are prime divisors of $|F|$ different from $q.$ Let $F_{q'}$ denote the $q'$-Hall subgroup of $F.$ Clearly we have $C_Q(F_{q'})=1$. Let now $y=\prod_{f\in F_{q'}}c^f$. Then we have $$1=y=(\prod_{f\in F_1\cap F_{q'} }c^f)(\prod_{f\in F_{q'}\setminus F_1}c^f)\in c^{|F_1\cap F_{q'}|}C_Q(U).$$ As a consequence  $c\in C_Q(U)$, because $q$ is coprime to $|F_{q'}|$. This contradiction establishes the claim. 
\end{proof}

{\it (6) We may assume that the group $FH$ is Frobenius-like of Type II.}

\begin{proof} On the contrary we assume that $FH$ is Frobenius-like of Type I.  By $\it(4),$ we have $H_1=Stab_H(\Omega_1)\ne 1$. Choose a transversal $T_1$ for $H_1$ in $H.$ Now $V=\bigoplus_{h\in T_1}W^h.$ Also we can guarantee the existence of a conjugate of $U$ which is $H_1$-invariant by means of the Schur-Zassenhaus Theorem as in $\it(5)$. There is no loss in assuming that $U$ is $H_1$-invariant.

  Set now $Y=\Sigma_{x\in F'}U^x$ and $F_2=Stab_F(Y)$ and $F_1=Stab_F(U).$  Clearly, $F_2=F'F_1$ and $Y$ is $H_1$-invariant.  Notice that  for all nonidentity $h\in H$, we have $C_F(h)\leq F'\leq F_2$ . Assume first that $F=F_2$. This forces that  we have  $V=Y$.  Clearly, $Y\ne U$, that is $F'\not \leq F_1$, because otherwise  $Q=[Q,F]=1$ due to the scalar action of the abelian group $Q$ on $U$.  So $F'\cap F_1=1$ which implies that $|F:F_1|$ is a prime. Then $F_1\unlhd F$ and $F'\leq F_1$ which is impossible. Therefore $F\ne F_2$. 
  
  If $1\neq h\in H$ and $t\in F$ such that ${Y}^{th}={Y}^{t}$ then $[h,t]\in F_2$. Now,  $F_2t=F_2t^h=(F_2t)^h$ and this implies the existence of an element in $F_2t\cap C_F(h)$. Since $ C_F(h)\leq F'\leq F_2$ we get $t\in F_2$. In particular, for each $t\in F\setminus F_2$ we have $Stab_{H}(Y^t)=1$. 

Let  $S$ be a transversal for $F_2$ in $F$. For any  $t\in S\setminus F_2$ set $Y_t=Y^t$ and consider $Z_t=\Sigma_{h\in H}{Y_t}^h$.  Notice that $V=Y\oplus \bigoplus_{t\in S\setminus F_2}Z_t$. As the sum $Z_t$ is direct we have $$C_{Z_t}(H)=\{ \sum_{h\in H}v^{h} \ : \ v\in Y_t\}$$ with $|C_{Z_t}(H)|=|Y_t|.$ Then $\mathbf r([Y_t,c])=\mathbf r([C_{Z_t}(H),c])\leq s$ for each $t\in S\setminus F_2$ with $[Y_t,c]\ne 0$.
On the other hand, $$ \Sigma\{\mathbf r([C_{Z_t}(H),c]) : t\in S \,\,\text{with}\,\,[Y_t,c]\ne 0\}\leq \mathbf r([C_V (H), c])\leq s$$   whence $|\{t\in S\setminus F_2 : [Y_t,c]\ne 0\}|$  is suitably bounded.  So the claim is established if there exists $t\in S\setminus F_2$ such that $[Y_t,c]\ne 0$, since we have $V=Y\oplus \bigoplus_{t\in S\setminus F_2}Z_t$.  Thus we may assume that $c$ is trivial on $\bigoplus_{t\in S\setminus F_2}Z_t$ and hence $[V,c]=[Y,c].$

There are two cases now: We have either $F'\cap F_1=1$ or $F'\leq F_1.$  First assume that $F'\leq F_1.$ Then we get $F_1=F_2$ because $F_2=F'F_1.$ Now $U=Y.$ Due to the  action by scalars of the abelian group $Q$ on $U$, it holds that $[Q,F_1]\leq C_Q(U)$. From this point on we can proceed as in the proof of step $\it(5)$ and observe that $C_Q(F_{q'})=1$. Letting now $y=\prod_{f\in F_{q'}}c^f$,  we have $$1=y=(\prod_{f\in F_1\cap F_{q'} }c^f)(\prod_{f\in F_{q'}\setminus F_1}c^f)\in c^{|F_1\cap F_{q'}|}C_Q(U).$$implying that  $c\in C_Q(U)$, because $q$ is coprime to $|F_{q'}|$.

Thus we have $F_1\cap F'=1$. First assume that $H_1=H$. Then $Y$ is $H$-invariant and   $F_1H$ is a  Frobenius group. Note that $C_U(F_1)=1$ as $C_V(F)=1$, and hence $C_{Y}(F_1)=1$  since $F'\leq Z(F).$ We consider now the action of $QF_1H$ on $Y$ and the fact that $\mathbf r([C_{Y}(H),C_Q(H)])\leq s.$ Then step $\it(5)$, we obtain that $\mathbf r(Y)=\mathbf r([Y,Q])$ is $(s,|H|)$-bounded. Next assume that $H_1\ne H.$ Choose a transversal for $H_1$ in $H$ and set $Y_1=\Sigma_{h\in T_1}Y^h$. Clearly this sum is direct and hence $$C_{Y_1}(H)=\{ \sum_{h\in T_1}v^{h} \ : \ v\in Y\}$$ with $|[C_{Y_1}(H),c]|=|[Y,c]|.$ Then $\mathbf r([Y,c])=\mathbf r([C_{Y_1}(H),c])\leq s$
establishing claim $\it(6)$.
 \end{proof}
 
 {\it (7) The proposition follows.}

\begin{proof} From now on $FH$ is a Frobenius-like group of Type II, that is, $H$ and $C_F(H)$ are of  prime orders.  By step $\it(4)$ we have  $H=H_1 =Stab_H( \Omega_1)$ since $|H|$ is a prime.  Now $V=W$. We may  also assume by the Schur-Zassenhaus theorem as in the previous steps that there is an $H$-invariant element, say $U$ in $\Omega$.  Let $T$ be a transversal for $F_1=Stab_F(U)$ in $F$.  Then  $F= \bigcup_{t\in T}{F_1}t$ implies  $V=\bigoplus_{t\in T}U^t$.  It should also be noted that we have $|\{t\in T : [U^t,c]\ne 0\}|$ is suitably bounded as $$[C_V(H),c]=\bigoplus \{ [C_{X_t}(H),c] : t\in T \,\, \text{with }\,\,[U^t,c]\ne 0\}\leq [C_V(H),C_Q(H)]$$ where $X_t=\bigoplus_{h\in H}U^{th}$.

Let $X$ be the sum of the components of all regular $H$-orbits on $\Omega$, and let $Y$ denote the sum of all $H$-invariant elements of $\Omega$. Then $V=X\oplus Y.$  Suppose that ${U}^{th}={U}^t$ for $t\in T$ and $1\ne h\in  H$. Now $[t,h]\in F_1$ and so the coset $F_1t$ is fixed by $H$. Since the orders of $F$ and $H$ are relatively prime we may assume that $t\in C_F(H).$ Conversely for each $t\in C_F(H)$,  ${U}^t$ is $H$-invariant. Hence the number of components in $Y$ is $|T\cap C_F(H)|=|C_F(H):C_{F_1}(H)|$ and so we have either $C_F(H)\leq F_1$ or not. 

If $C_F(H)\not\leq F_1$ then $C_{F_1}(H)=1$ whence $F_1H$ is Frobenius group acting on $U$ in such a way that $C_{U}(F_1)=1$.  Then  $\mathbf  r(U)$ is $(s,|H|)$-bounded by step $\it(5)$ since  $\mathbf  r([C_{U}(H),C_Q(H)])\leq s$ holds. This forces that $\mathbf r([V,c])$ is bounded suitably and hence the claim is established.

Thus we may assume that $C_F(H)\leq F_1.$ Then  $Y=U$ is the unique $H$-invariant $Q$-homogeneous component.  If $[U^t,c]\ne 0$ for some $t\in F\setminus F_1$ we can bound $\mathbf r(U)$ and hence $\mathbf r([V,c])$ suitably. Thus we may assume that $c$ is trivial on $U^t$ for each $t\in F\setminus F_1.$ Due to the  action of the abelian group $Q$ on $U$, it holds that $[Q,F_1]\leq C_Q(U)$.  From this point on we can proceed as in the proof of step $\it(5)$ and observe that $C_Q(F_{q'})=1$. Letting now $y=\prod_{f\in F_{q'}}c^f$,  we have $$1=y=(\prod_{f\in F_1\cap F_{q'} }c^f)(\prod_{f\in F_{q'}\setminus F_1}c^f)\in c^{|F_1\cap F_{q'}|}C_Q(U).$$implying that  $c\in C_Q(U)$, because $q$ is coprime to $|F_{q'}|$.  This final contradiction completes the proof of Proposition 3.1.
\end{proof}
\end{proof}

The next proposition  studies the action of a dihedral group of automorphisms  and is essential in proving Theorem B. 

\begin{proposition}\label{prop2}
Let $D= \langle\alpha, \beta \rangle$ be a dihedral group generated by two involutions $\alpha$ and $\beta$. Suppose that $D$ acts on a $q$-group $Q$ for some prime $q$. Let $V$ be an irreducible $\mathbb{F}_pQD$-module where $\mathbb{F}_p$ is a field with characteristic $p$ not dividing $|Q|$. Suppose that $C_{VQ}(F)=1$ where $F=\langle \alpha\beta\rangle$. If $max\{\mathbf r([C_V(\alpha), C_Q(\alpha)]),\mathbf r([C_V(\beta), C_Q(\beta)])\}\leq s$, then $\mathbf r([V,Q])$ is $s$-bounded. 
\end{proposition}

\begin{proof} 

We set $H=\langle \alpha\rangle$.  So $D=FH$.  By Lemma \ref{sink1} and Theorem \ref{theDih}, we have $[V,Q]=[V, C_Q(\alpha)][V,C_Q(\beta)]$. Then it is sufficient to bound the rank  of $[V,C_Q(H)]$. Following the same steps as in the proof of Proposition \ref{prop1} by replacing Proposition 2.3 by Proposition 2.4,  we observe that $Q$ acts faithfully on $V$ and $Q=\langle c^F\rangle$ is abelian with $c\in C_{Z(Q)}(H)$ of order $q$. Furthermore $Ker(C_Q(H) \ \textrm{on} \ C_V(H)) = Ker(C_Q(H) \ \textrm{on} \ V)=1$. Note that it suffices to bound $\mathbf r([V,c])$ suitably.

Let $\Omega$ denote the set of  $Q$-homogeneous components of the irreducible $QD$-module $V.$ Let $\Omega_1$ be  an $F$-orbit of $\Omega$ and set $W=\Sigma_{U\in \Omega_1}U.$ Then we have $V=W+W^{\alpha}$.  Suppose that $W^{\alpha}\ne W$. Then for any $U\in \Omega_1$ we have $Stab_H(U)=1$.  Let  $T$ be a tranversal for $Stab_F(U)=F_1$ in $F$ . It holds that  $V=\Sigma_{t\in T}X_t$ where $X_t=U^t+U^{t\alpha}.$  Now $[V,c]=\Sigma_{t\in T}[X_t,c]$ and $C_V(H)=\Sigma_{t\in T} C_{X_t}(H)$ where $C_{X_t}(H)=\{w+w^{\alpha} : w\in U^t\}$.   Since $[V,c]\ne 0$ there exists $t\in T$ such that $[U^t,c]\ne 0$, that is $[U^t,c]=U^t.$  Then $[C_{X_t}(H),c]=C_{X_t}(H).$ Since  $\mathbf  r([C_V(H),C_Q(H)])\leq s$ we get $\mathbf  r(U)=\mathbf  r(C_{X_t}(H))\leq s$.  Furthermore it follows that $|\{t\in T : [U^t,c]\ne 0\}|$ is $s$-bounded and as a consequence $\mathbf  r([V,c])$ is suitably bounded.  Thus we may assume that $W^{\alpha}=W$ which  implies that $\Omega_1=\Omega$ and $H$ fixes an element, say $U$, of $\Omega$ as desired.

 Let $U^t\in \Omega$ be $H$-invariant. Then $[t,\alpha]\in F_1.$ On the other hand $t^{-1}t^{\alpha}=t^{-2}$ since $\alpha$ inverts $F$.  So $F_1t$ is an element of $F/F_1$ of order at most $2$ which implies that the number of $H$-invariant elements of $\Omega$ is at most $2$.  Let now $Y$ be the sum of all $H$-invariant elements of $\Omega$. Then $V=Y\oplus \bigoplus_{i=1}^m X_i$ where  $X_1,\ldots X_m$ are the sums of elements in  $H$-orbits of length $2.$  Let $X_i=U_i\oplus U_i^{\alpha}$. Notice that if $[U_i,c]\ne 0$ for some $i$, then we obtain $\mathbf  r(U)=\mathbf  r(U_i)\leq s$ by a similar argument as above. On the other hand we observe that the number of $i$ for which $[U_i,c]\ne 0$ is $s$-bounded by the the hypothesis that $\mathbf r([C_V(H),c])\leq s$.  It follows now that $\mathbf r([V,c])$ is suitably bounded in case where $[U_i,c]\ne 0$ for some $i$. 

Thus we may assume that $c$ centralizes  $\bigoplus_{i=1}^m X_i$ and that $[U,c]=U$.  Due to the  scalar action by scalars of the abelian group $Q$ on $U$, it holds that $[Q,F_1]\leq C_Q(U)$. As $F_1\unlhd FH$, we have $[Q,F_1]\leq  C_Q(V)=1$. Clearly we have $C_Q(F_{q'})=1$ where $F_{q'}$ denotes the Hall $q'$-part of $F$ whose existence is guaranteed by the fact that $C_Q(F)=1.$ Let now $y=\prod_{f\in F_{q'}}c^f$. Then we have $$1=y=(\prod_{f\in F_1\cap F_{q'} }c^f)(\prod_{f\in F_{q'}\setminus F_1}c^f)\in c^{|F_1\cap F_{q'}|}C_Q(U).$$ As a consequence  $c\in C_Q(U)$, because $q$ is coprime to $|F_{q'}|$. This contradiction completes the proof of Proposition \ref{prop2}.

\end{proof}

\section{Proofs of theorems}

Firstly, we shall give a detailed proof  for Theorem A part (b). The proof of Theorem A (a)  can be easily obtained by just obvious modifications of the proof of part (b).

First, we assume that $G = PQ$ where $P$ and $Q$ are $FH$-invariant subgroups such that $P$ is a normal $p$-subgroup for a prime $p$ and $Q$ is a nilpotent $p'$-group with $|[C_P(H),C_Q(H)]|=p^s$. We shall prove that $\mathbf r(\gamma_{\infty}(G))$ is $((s,|H|)$-bounded. Clearly $\gamma_{\infty}(G)=[P,Q]$. Consider an unrefinable $FH$-invariant normal series $$P=P_{1}>P_{2}>\cdots>P_{k}>P_{k+1}=1.$$ Note that its factors $P_i/P_{i+1}$ are elementary abelian. Let $V=P_{k}$.  Since $C_V(Q)=1$, we have that $V=[V,Q]$. We can also assume that $Q$ acts faithfully on $V$. Proposition \ref{prop1} yields  that $\mathbf r(V)$ is $(s, |H|)$-bounded.  Set $S_i=P_i/P_{i+1}$. If $[C_{S_i}(H),C_Q(H)]=1$, then $[S_i,Q]=1$ by Proposition  \ref{Fr}. Since $C_P(Q)=1$ we conclude that each factor $S_i$ contains a nontrivial  image of an element of $[C_P(H),C_Q(H)]$. This forces that $k \leq s$.  Then we proceed by induction on $k$ to obtain that  $\mathbf r([P,Q])$ is an $(s,|H|)$-bounded number, as desired. 

Let $F(G)$ denote the Fitting subgroup of a group $G$. Write $F_{0}(G)=1$ and let $F_{i+1}(G)$ be the inverse image of $F(G/F_{i}(G))$. As is well known, when  $G$ is soluble, the least number $h$ such that $F_{h}(G)=G$ is called the Fitting height $h(G)$ of $G$. Let now $r$ be the rank of $\gamma_{\infty}(C_G(H))$. Then $C_G(H)$ has $r$-bounded Fitting height (see for example Lemma 1.4 of \cite{KS2}) and hence $G$ has $(r,|H|)$-bounded Fitting height. 

We shall proceed by  induction on $h(G)$. Firstly, we consider the case where $h(G)=2$.  Indeed, let $P$ be a Sylow $p$-subgroup of $\gamma_{\infty}(G)$ and $Q$ an $FH$-invariant Hall $p'$-subgroup of $G$. Then, by the preceeding paragraphs and Lemma \ref{lema 3.1}, the rank of $P=[P,Q]$ is $(r,|H|)$-bounded and so the rank of $\gamma_{\infty} (G)$ is $(r,|H|)$-bounded. Assume next that $h(G)>2$ and let $N=F_2(G)$ be the second term of the Fitting series of $G$. It is clear that the Fitting height of $G/\gamma_{\infty} (N)$ is $h-1$ and $\gamma_{\infty} (N)\leq \gamma_{\infty}(G)$. Hence, by induction we have that $\gamma_{\infty}(G)/\gamma_{\infty} (N)$ has $(r,|H|)$-bounded rank. As a consequence, it holds that  $${\bf r}(\gamma_{\infty}(G))\leq {\bf r}( \gamma_{\infty}(G)/\gamma_{\infty} (N))+{\bf r}(\gamma_{\infty}(N))$$ completing the proof of Theorem A(b).

The proof of Theorem B can be directly obtained as in the above argument by replacing Proposition \ref{prop1} by Proposition \ref{prop2}; and  Proposition \ref{Fr} by Proposition \ref{dih}.

\end{document}